\documentclass[10pt]{amsart}
\usepackage{amssymb}
\usepackage{bm}
\usepackage{graphicx}
\usepackage[centertags]{amsmath}
\usepackage{amsfonts}
\usepackage{amsthm}
\usepackage{enumerate}
\usepackage{xcolor}

\newtheorem{theorem}{Theorem}
\newtheorem*{theorem*}{Theorem}

\newtheorem{corollary}[theorem]{Corollary}
\newtheorem{lemma}[theorem]{Lemma}
\newtheorem{rem}[theorem]{Remark}
\newtheorem{notation}[theorem]{Notation}
\newtheorem{proposition}[theorem]{Proposition}

\newtheorem{problem}{Problem} 
\theoremstyle{definition}
\newtheorem{definition}[theorem]{Definition}
\newenvironment{remark}[1][Remark]{\begin{trivlist}
\item[\hskip \labelsep {\bfseries #1}]}{\end{trivlist}}

\newcommand{\fin}{[\mathbb{N}]^{<\omega}}
\newcommand{\infin}{[\mathbb{N}]}
\newcommand{\mcg}{\mathcal{G}}

\newcommand{\mcf}{\mathcal{F}}
\newcommand{\mcs}{\mathcal{S}}

\newcommand{\nn}{\mathbb{N}}

\newcommand{\ee}{\varepsilon}

\newcommand{\inn}{\in \mathbb{N}}

\newcommand{\supp}{\mathrm{supp}}

\newcommand{\sspan}{\mathrm{span}}
\DeclareMathOperator{\sgn}{sgn}

\newcommand{\e}{\varepsilon}
\newcommand{\al}{\alpha}

\newcommand{\schr}{Schreier functional}
\newcommand{\la}{\lambda}

\newcommand{\X}{\mathfrak{X}_{_{^{0,1}}}^n}
\newcommand{\Xxi}{\mathfrak{X}_{_{^{0,1}}}^\xi}

\newcommand{\Xox}{\mathfrak{X}_{_{^{0,1}}}^{\omega^\xi}}

\DeclareMathOperator{\ran}{ran}

\begin{document}
 
\title [Arbitrarily distortable Banach spaces of higher order]{Arbitrarily distortable Banach spaces of higher order}

\author{Kevin Beanland}
\address{Department of Mathematics, Washington and Lee University, Lexington, VA 24450.}
\email{beanlandk@wlu.edu}

\author{Ryan Causey}
\address{Department of Mathematics and Computer Science\\
Texas A\&M University, College Station, TX 77845  USA}
\email{rcausey@math.tamu.edu}

\author[Pavlos Motakis]{Pavlos Motakis}
\address{National Technical University of Athens, Faculty of Applied Sciences,
Department of Mathematics, Zografou Campus, 157 80, Athens,
Greece} \email{pmotakis@central.ntua.gr}

\thanks{2010 \textit{Mathematics Subject Classification}. Primary: 46B03, 46B06}
\thanks{\textit{Key words:} Spreading models, distortion}
\thanks{The first named author acknowledges support from the Lenfest Summer Grant at Washington and Lee University.}
\thanks{The first and third authors would like to acknowledge the support of program API$\Sigma$TEIA-1082.}

\begin{abstract}
We study an ordinal rank on the class of Banach spaces with bases that quantifies the distortion of the norm of a given Banach space. The rank $AD(\cdot)$, introduced by P. Dodos, uses the transfinite Schreier familes and has the property that $AD(X) <\omega_1$ if and only if $X$ is arbitrarily distortable. We prove several properties of this rank as well as some new results concerning higher order $\ell_1$ spreading models. We also compute this rank for for several Banach spaces. In particular, it is shown that class of Banach spaces $(\Xox)_{\xi  <\omega_1}$, which each admit $\ell_1$ and $c_0$ spreading models hereditarily, and were introduced by S.A. Argyros, the first and third author, satisfy $AD(\Xox)=\omega^\xi+1$. This answers some questions of Dodos.  
\end{abstract}
\maketitle

\section{Introduction}

Let $(X,\|\cdot\|)$ be a Banach space with a Schauder basis $(e_i)_{i \inn}$ and $t>1$. We say that $X$ is $t$-distortable if there is an equivalent norm $|\cdot |$ on $X$ so that for each normalized block sequence $(x_n)$ of $(e_i)$ there is a finite set $F \subset \nn$ and vectors $x,y \in \sspan \{x_n: n \in F\}$ so that 
$$\|x\|=\|y\|=1 \mbox{ and } \frac{|x|}{|y|}>t.$$
A space is arbitrarily distortable if it is $t$-distortable for each $t>1$. 

In the 1960s, R.C. James \cite{JamesNonSquare}
proved that $\ell_1$ and $c_0$ are not $t$-distortable for any $t>1$. In 1994, E. Odell and Th. Schlumprecht \cite{OSdistortion} famously proved that the spaces $\ell_p$, for $1<p< \infty$ are arbitrarily distortable. Whether there is a space that is distortable for some $t>1$ but not arbitrarily distortable is a central open problem in Banach space theory \cite{Gowersblog}. Other important results on distortion can be found in the references \cite{AD,Maureydistortion,MilmanTJdistortion,TJ-GAFA}

In the current paper we study distortion in Banach spaces from a different point of view. Instead of asking whether a given space is $t$-distortable, we consider the problem of quantifying, by using the transfinite Schreier families, the complexity of the distortion. In particular we would like to know how `difficult' it is to find the finite set $F$ that witnesses the distortion in the above definition. Following P. Dodos \cite{DodosMO}, if we consider a collection $\mathcal{G}$ of finite subsets of $\nn$ we say that a space $X$ with a basis is $t$-$\mathcal{G}$ distortable if the set $F$, in the definition of $t$-distortable, can be choosen as an element of $\mathcal{G}$. If for any $\mathcal{G}$ a space is $t$-$\mathcal{G}$ distortable it must be $t$-distortable. A space is $\mathcal{G}$ arbitrarily distortable if it is $t$-$\mathcal{G}$ distortable for all $t>1$. We study the cases where $\mathcal{G}$ is a Schreier family $\mathcal{S}_\xi$ for some countable ordinal $\xi$. This naturally gives rise to a ordinal rank on a space; namely, the minimum $\xi$ so that $X$ is $t$-$\mathcal{S}_\xi$ distortable. 

In this paper we record the definition of this ordinal rank and some facts concerning it (see Proposition \ref{propindex}). In particular we prove that a space $X$ with a basis is arbitrarily distortable if and only if there is a countable ordinal $\xi$  so that $X$ is $\mathcal{S}_\xi$ arbitrarily distortable. We also answer some natural questions raised by P. Dodos \cite{DodosMO}. In particular, we prove the following:

\begin{theorem*}
For each countable ordinal $\xi$ there is a reflexive space $\Xox$ so that for every block subspace $X$ of $\Xox$ we have $AD(X) = \omega^\xi +1$. Moreover, every subspace of this space contains a $c_0^1$ and an $\ell_1^{\omega^\xi}$ spreading model.  
\end{theorem*}

The spaces $\Xox$ are introduced in a recent paper of S.A. Argyros and the first and third authors \cite{ABM-Illinois}. We prove several of the properties of these spaces in the final section of the paper.

As a step towards showing that $AD(\Xox)>\omega^\xi$ we prove the following result concerning $\ell_1^{\omega^\xi}$ that we believe is of independent interest.

\begin{theorem*} Let $X$ be a Banach space, $\xi< \omega_1$.  If $X$ contains an $\ell_1^{\omega^\xi}$ spreading model, then for any $\varepsilon>0$, $X$ contains a $(1+\varepsilon)$-$\ell_1^{\omega^\xi}$ spreading model. Moreover the same result holds replacing $\ell_1^{\omega^\xi}$ with $c_0^{\omega^\xi}$  
\end{theorem*}

The above theorem is analogous to a result concerning block indices proved by Judd and Odell \cite{JuddOdell} and extends Remark 6.6 (iii) found in this paper. We also compute certain distortion indices for several other Banach space including Tsirelson space and Schlumprecht space \cite{Sc}. Our computations rely heavily on the presence of $\ell_1$ and $c_0$ structure in our spaces and uses James' well-known blocking arguments. Consequently, our methods do not allow us to compute lower bounds for spaces lacking this type of structure.

Finally, we note that W.T. Gowers asked if $\ell_2$ is $t$-$\mathcal{S}_1$ distortable for any $t>1$ \cite{DodosMO}. As he noted in the given reference, this problem can be interpreted as a distortion variant of the strengthened Finite Ramsey Theorem. All proofs of the strengthened finite Ramsey Theorem use the infinite Ramsey Theorem and, indeed, the strengthened Finite Ramsey Theorem is unprovable in Peano Arithmetic \cite{PH}. On the other hand, Gowers showed in \cite{Go1} that the infinite Ramsey Theorem is false in the Banach space setting and, consequently, this problem is likely to be very difficult or perhaps, in an extreme case, undecidable.

This paper is organized as follows. In section 2 we set our notation, give basic definitions and facts concerning Schreier families. Section 3 contains the precise definition of the distortion index and some general facts concerning this index. In section 4 the second theorem listed above is proved. The technique for this proof is then used to prove corresponding results concerning the distortion indices for spaces admitting $\ell_1^{\omega^\xi}$ or $c_0^{\omega^\xi}$ spreading models. It is also shown that admitting no $\ell_1^{\omega^\xi}$ 
spreading model is a three space property.  In section 5 we use the results from section 4 to compute some distortion indices for certain spaces. The final section contains some previously unpublished facts about the spaces $\mathfrak{X}_{_{0,1}}^{\omega^\xi}$ which first appeared in \cite{ABM-Illinois} and which are needed to prove that $AD(\mathfrak{X}_{_{0,1}}^{\omega^\xi})=\omega^\xi+1$.

\section{Notation, Schreier families, basic facts}

\subsection{Notation and terminology}
We will often begin with a Banach space $X$ having norm $\|\cdot\|$ and consider an equivalent norm $|\cdot|$ on $X$.  If we write $S_X$ or refer to normalization of a vector without specifying a norm, it is with respect to the norm $\|\cdot\|$.  

Throughout, if $M$ is any subset of $\mathbb{N}$, we let $[M]^{<\omega}$ and $[M]$ denote the finite and infinite subsets of $M$, respectively.  We will identify subsets of the natural numbers in the obvious way with strictly increasing sequences of natural numbers.  We write $E<F$ if $\max E<\min F$, $n<F$ if $n<\min F$, and $n\leqslant F$ if $n\leqslant \min F$.  We follow the convention that $\min \varnothing = \omega$, $\max \varnothing = 0$.  If $(E_i)$ is a (finite or infinite) sequence in $[\mathbb{N}]^{<\omega}$ satisfying $E_i<E_{i+1}$ for all $i\in \mathbb{N}$, we call the sequence $(E_i)$ {\it successive}.  If $E\in \fin$ and $x\in \ell_\infty$, we let $Ex\in \ell_\infty$ be the sequence so that $Ex(i)=x(i)$ if $i\in E$ and $Ex(i)=0$ otherwise.    

If $(m_i)_{i\in I}, (n_i)_{i\in I}$ are (finite or infinite) strictly increasing subsequences in $\mathbb{N}$ with the same length so that $m_i\leqslant n_i$ for all $i\in I$, we say $(n_i)_{i\in I}$ is a {\it spread} of $(m_i)_{i\in I}$.  We say a subset $\mcf\subset  \fin$ is {\it spreading} if it contains all spreads of its members.  We say $\mcf$ is {\it hereditary} if it contains all subsets of its members.  We let $\mathfrak{S}$ denote the set of all non-empty, spreading, hereditary subsets of $\fin$.  

If $E\in \fin$ and $M=(m_i)\in [\mathbb{N}]$, we let $M(E)=(m_i: i\in E)$.  If $\mcf\subset \fin$, we let $\mcf(M) = \{M(E): E\in \mcf\}$.   If $\mcf, \mcg\subset\fin$, we let $$\mcf[\mcg]= \Bigl\{\underset{i=1}{\overset{n}{\bigcup}}E_i: E_1<\ldots<E_n, E_i\in \mcg \text{\ }\forall i, (\min E_i)_{i=1}^n\in \mcf\Bigr\}.$$  
It is easily checked that $(\mcf, \mcg)\mapsto \mcf[\mcg]$ defines an associative operation from $\mathfrak{S}^2$ into $\mathfrak{S}$.

If $(e_i)$ is a Schauder basic sequence with coordinate functionals $(e_i^*)$ and if  $x\in [e_n]:=\overline{\sspan\{e_n: n\inn\}}$, we let $\text{supp}_{(e_i)} x=(i:e_i^*(x)\neq 0)$.  When the basis is understood, we will write $\text{supp} x$ in place of $\text{supp}_{(e_i)}x$.  If $x,y\in [e_i]$ are such that $\text{supp}_{(e_i)} x< \text{supp}_{(e_i)} y$, we write $x<y$.  

\subsection{Schreier families}
We define for each $\xi< \omega_1$ the Schreier family  $\mcs_\xi\in \mathfrak{S}$ \cite{AA}.  The purpose of these families is to measure complexity, which will be made precise below.  We let $$\mcs_0 = \{\varnothing\}\cup \bigl\{(n): n\in \mathbb{N}\bigr\},$$ $$\mcs_1=\{E: |E|\leqslant \min E\},$$ $$\mcs_{\xi+1} = \mcs_1[\mcs_\xi],$$ and if $\mcs_\zeta$ has been defined for each $\zeta<\xi$, $\xi<\omega_1$ a limit ordinal, we choose $\xi_n\uparrow \xi$ and let $$\mcs_\xi = \{E: \exists n\leqslant E\in \mcs_{\xi_n}\}.$$  
One can easily show by induction that in the limit ordinal case, the sequence $\xi_n\uparrow \xi$ can be chosen so that for each $i$, $\mcs_{\xi_i}\subset \mcs_{\xi_{i+1}}$.  It will be convenient for us to proceed with this assumption. Note that $\mcs_1\subset \mcs_\xi$ for all $\xi\geqslant1$.    

For each natural number $n$, we let $$\mathcal{A}_n=\{E\in \fin: |E|\leqslant n\}.$$

We will use the following facts about the Schreier families, which are related to or contained in  \cite{OdellTJWagner}:

\begin{proposition}\begin{enumerate}
\item[(i)]  If $\xi\leqslant \zeta$, there exists $n\in \mathbb{N}$ so that if $n\leqslant E\in \mcs_\xi$, $E\in \mcs_\zeta$.  
\item[(ii)] For any $0\leqslant \xi, \zeta<\omega_1$ and $M\in [\nn]$, there exists $L\in [M]$ so that $$\mcs_\xi(L)[\mcs_\zeta]\subset \mcs_{\zeta+\xi}.$$ 
The above inclusion holds if we replace $L$ by any spread of $L$.

 \item[(iii)] If $\xi, \zeta< \omega_1$, there exists $L\in \infin$ so that $\mcs_\xi[\mcs_\zeta](L) \subset \mcs_{\zeta+\xi}$, and this inclusion holds if we replace $L$ by any spread of $L$. 

\item[(iv)] For $0\leqslant \xi, \zeta <\omega_1$ and a successive sequence $(F_i)$ of members of $\mcs_\zeta$, there exists $N\in [\nn]$ so that if $E\in \mcs_\xi(N)$, $\cup_{i\in E}F_i\in \mcs_{\zeta+\xi}$.

\end{enumerate}
\label{Schreier facts}
\end{proposition}

Items (i) and (iii) are contained in \cite{OdellTJWagner}. To the best of our knowledge, item (ii) has not appeared in the literature. The proof of item (ii) is similar to the proof of item (iii), however, since it is new and a somewhat complicated, we include it for completeness.

\begin{proof}
(i) This item is contained in \cite{OdellTJWagner}. 

(ii) First, we note that if $L\in [M]$ has been chosen so that $$\mcs_\xi(L)[\mcs_\zeta]\subset \mcs_{\zeta+\xi},$$ we can replace $L$ with any $L'\in [L]$ and still have the desired containment.  

We fix $\zeta$ and prove the result by induction on $\xi$.  Since $$\mcs_0(L)[\mcs_\zeta]=\{E\in \mcs_\zeta: \min E\in L\}\subset \mcs_\zeta,$$ we can take $L=M$ in the base case.  

Suppose $M\in [\mathbb{N}]$ is given and $L\in [M]$ is such that $\mcs_\xi(L)[\mcs_\zeta]\subset \mcs_{\zeta+\xi}$. Then $\mcs_{\xi+1}(L)[\mcs_\zeta]\subset \mcs_{\zeta+\xi+1}$.  To see this, take $$E=\bigcup_{i=1}^n E_i\in \mcs_{\xi+1}(L)[\mcs_\zeta],$$ where $E_1<\ldots<E_n$, $E_i\in \mcs_\zeta$ for each $i$, and $A=(\min E_i)_{i=1}^n\in \mcs_{\xi+1}(L)$.  Let $B\in \mcs_{\xi+1}$ be such that $L(B)=A$.  Write $$B=\bigcup_{j=1}^k B_j,$$ where $B_1<\ldots<B_k$, $B_j\in \mcs_\xi$ for each $j$, and $k\leqslant B$.  Let $$I_j=(i\leqslant n: \min E_i\in L(B_j))$$ and $F_j=\cup_{i\in I_j}E_i$.    Note that $(\min F_i)_{i\in I_j}= L(B_j)\in \mcs_\xi(L)$, so $F_j\in \mcs_\xi(L)[\mcs_\zeta]\subset \mcs_{\zeta+\xi}$.  Moreover, $F_1<\ldots<F_k$ and $$\min F_1 =\min E \geqslant\min B\geqslant k,$$ so $$E=\bigcup_{j=1}^k F_j \in \mcs_{\zeta+\xi+1}.$$  

Last, suppose the result holds for each $\gamma<\xi$, $\xi<\omega_1$ a limit ordinal.  Note that $\zeta+\xi$ is also a limit ordinal.  Let $\xi_n\uparrow\xi$, $\gamma_n\uparrow \zeta+\xi$ be the ordinals used to define $\mcs_\xi$ and $\mcs_{\zeta+\xi}$, respectively.  Recall that we have selected these so that $\mcs_{\xi_n}\subset \mcs_{\xi_{n+1}}$ for all $n\in \mathbb{N}$.  First choose natural numbers $k_n$ so that $\zeta+\xi_n<\gamma_{k_n}$ for each $n\in \mathbb{N}$.  Next, choose natural numbers $r_n\geqslant k_n$ so that if $r_n\leqslant E\in \mcs_{\zeta+\xi_n}$, then $E\in \mcs_{\gamma_{k_n}}$.  Define $L_0=M$ and choose recursively $L_1, L_2, \ldots$ so that $L_n\in [L_{n-1}]$, $r_n\leqslant L_n$, and $\mcs_{\xi_n}(L_n)[\mcs_\zeta]\subset \mcs_{\zeta+\xi_n}$ for each $n\in \mathbb{N}$.  Let $L_n=(\ell^n_i)_i$ and let $\ell_n=\ell^n_n$.  Note that $\ell_1<\ell_2<\ldots$, and let $L=(\ell_n)$.  Fix $$E=\bigcup_{i=1}^k E_i\in \mcs_\xi(L)[\mcs_\zeta],$$ $E_1<\ldots<E_k$, $E_i\in \mcs_\zeta$, $A=(\min E_i)_{i=1}^k\in \mcs_\xi(L)$.  Choose $B\in \mcs_\xi$ so that $L(B)=A$ and set $n=\min B$.  Then $B\in \mcs_{\xi_n}$.  Let $L'=(\ell_1^n, \ldots, \ell^n_n, \ell_{n+1}, \ldots)$ and note that $L'\in [L_n]$.  Moreover, $A=L(B)=L'(B)$, so $$E\in \mcs_{\xi_n}(L')[\mcs_\zeta]\subset \mcs_{\xi_n}(L_n)[\mcs_\zeta]\subset\mcs_{\zeta+\xi_n}.$$  Furthermore, $$\min E\geqslant \min L_n \geqslant r_n \geqslant k_n,$$ so $$E\in \mcs_{\zeta+\xi_n}\cap [r_n, \infty)^{<\omega}\subset \mcs_{\gamma_{k_n}}\cap [k_n, \infty)^{<\omega}\subset \mcs_{\zeta+\xi}.$$  

(iii) This follows from (ii).  Let $M=\nn$ and choose $L\in [M]$ to satisfy the conclusion of (ii).  Then $$\mcs_\xi[\mcs_\zeta](L)\subset \mcs_\xi(L)[\mcs_\zeta]\subset \mcs_{\zeta+\xi}.$$  The containment still holds if we replace $L$ by any spread $L'$ of $L$, since in this case the elements of $\mcs_\xi[\mcs_\zeta](L')$ are spreads of elements of $\mcs_\xi[\mcs_\zeta](L)$ and $\mcs_{\zeta+\xi}$ is spreading.  

(iv) Let $m_i=\min F_i$ and $M=(m_i)$.  Choose $L=(m_{n_i})\in [M]$ so that $\mcs_\xi(L)[\mcs_\zeta]\subset \mcs_{\zeta+\xi}$.  We claim $N=(n_i)$ satisfies the conclusion.  If $E=N(F)\in \mcs_\xi(N)$, $$(\min F_i)_{i\in E} = (m_i)_{i\in E} = (m_{n_i})_{i\in F}=L(F)\in \mcs_\xi(L).$$  Then $$\bigcup_{i\in F} F_i\in \mcs_\xi(L)[\mcs_\zeta]\subset \mcs_{\zeta+\xi}.$$  \end{proof}

\section{Distortion indices}

Recall that if $t>1$, we say a Banach space $X$ with basis $(e_i)$ is $t$-{\it distortable}  if there exists an equivalent norm $|\cdot|$ on $X$ so that for any block sequence $(x_i)$ in $X$ there exists $F\in \fin$ and $x,y\in [x_i]_{i\in F}$ with $\|x\|=\|y\|=1$ and $|x|/|y|>t$. It is easy to see that if a Banach space with a basis is $t$-distortable with this definition then it is ($t-\delta$)-distortable for each $\delta>0$ using the usual definition of distortion. We say $X$ is {\it arbitrarily distortable} if it is $t$-distortable for all $t> 1$.

Let $\mcf\in \mathfrak{S}$, $t\geqslant 1$.  If $X$ is a Banach space with basis $(e_i)$, we will say an equivalent norm $|\cdot|$ on $X$ is a $t$-$\mcf$ {\it distortion} of $X$ if for all normalized blocks $(x_i)$ of $(e_i)$, there exists $E\in \mcf$ and $x,y\in [x_i]_{i\in E}$ with $\|x\|=\|y\|=1$ and $|x|/|y|>t$.  We say $X$ is $t$-$\mcf$ {\it distortable} if there exists a $t$-$\mcf$ distortion of $X$.  We say $X$ is $\mcf$ {\it arbitrarily distortable} if it is $t$-$\mcf$ distortable for every $t\geqslant 1$.  We let $$D_t(X) = \min \{\xi< \omega_1: X \text{\ is\ }t\text{-}\mcs_\xi \text{\ distortable}\}$$ if this set is non-empty, and $D_t(X)=\omega_1$ otherwise.  We let $$AD(X) = \min\{\xi< \omega_1: X\text{\ is \ }\mcs_\xi \text{\ arbitrarily distortable}\}$$ if this set is non-empty, and $AD(X) = \omega_1$ otherwise.  Formally speaking, these indices should refer to the basis, since it is not true {\it a priori} that this index is independent of the choice of basis, but we will abuse notation and assume the basis is understood.  

In the next proposition we record some relevant facts concerning this index. In particular, we observe that a space $X$ is arbitrarily distortable if and only if $AD(X)<\omega_1$. P. Dodos makes this observation in \cite{DodosMO}. 

\begin{proposition}Let $X$ be a Banach space with basis $(e_i)$ and let $t> 1$.  
\begin{enumerate}[(i)] 
\item The space $X$ is not $t$-distortable if and only if $D_t(X)=\omega_1$. 
\item If $D_t(X)=\xi<\omega_1$ for some $t$, then $X$ is $t$-$\zeta$ distortable for any $\xi<\zeta<\omega_1$.  
\item $AD(X)= \underset{n\in \mathbb{N}}{\sup} D_n(X)$. 
\item $X$ is arbitrarily distortable if and only if $AD(X)<\omega_1$.  \end{enumerate}\label{propindex}
\end{proposition}

\begin{proof}

$(i)$ Suppose $X$ is not $t$-distortable.  Let $|\cdot|$ be any equivalent norm on $X$.  Then by definition, there must exist a block sequence $(x_i)$ in $X$ so that for each $E\in \fin$ and each $x,y\in [x_i]_{i\in E}$ with $\|x\|=\|y\|=1$, $|x|/|y|\leqslant t$.  Then for any $\xi<\omega_1$, the sequence $(x_i)$ witnesses the fact that $X$ is not $t$-$\mcs_\xi$ distortable.  

In the reverse direction, suppose $D_t(X)=\omega_1$.  Let $|\cdot|$ be any equivalent norm on $X$.  For each $\xi< \omega_1$ there exists a normalized block $(x^\xi_i)_i$ so that for $E\in \mcs_\xi$ and $x,y\in [x^\xi_i]_{i\in E}$ with $\|x\|=\|y\|=1$, $|x|/|y|\leqslant t$.  Let $$T=\Bigl\{(x_i)_{i=1}^n\subset S_X: x_1<\ldots<x_n, |x|\|y\|\leqslant t|y|\|x\| \text{\ \ }\forall x,y\in [x_i]_{i=1}^n\Bigr\}.$$  One easily checks that $(k_1, \ldots, k_n)\mapsto (x^\xi_{k_i})_{i=1}^n$ is a tree isomorphism of $\mcs_\xi\setminus (\varnothing)$ with a subtree of $T$.  This means the order $o(T)=\omega_1$.  Since $X$ is separable and $T$ is clearly a closed tree, using Bourgain's version of the Kunen-Martin Theorem \cite{BourgainBelgium}, there must exist an infinite branch, say $(x_i)$.  Clearly $Y=[x_i]$ is so that if $x,y\in S_Y$, $|x|/|y|\leqslant t$.  

$(ii)$ Suppose $|\cdot|$ is a $t$-$\mathcal{S}_\xi$ distortion on $X$.  Let $(x_i)$ be a normalized block sequence in $X$.  By Proposition $1$ (i), there exists $n$ so that if $n\leqslant E\in \mcs_\xi$, $E\in \mcs_\zeta$.  We apply the definition of $t$-$\mathcal{S}_\xi$ distortion to the block sequence $(x_{i+n})_i$ to deduce the existence of $E\in \mcs_\xi$ and $x,y\in [x_{i+n}]_{i\in E}$ so that $\|x\|=\|y\|=1$ and $|x|/|y|>t$.  Then letting $F=(i+n:i\in E)\in \mcs_\zeta$, we deduce that $x,y$ witness the fact that $|\cdot|$ is also a $t$-$\mathcal{S}_\zeta$ distortion on $X$.

$(iii)$ Clearly $AD(X)\geqslant \sup D_n(X)$.  If $D_n(X)=\omega_1$ for some $n$, the result is clear.  So assume $D_n(X)<\omega_1$ for each $n$ and let $\xi=\sup D_n(X)\in\omega_1$.  Then by $(iii)$, $X$ is $n$-$\mathcal{S}_\xi$ distortable for each $n$, and $AD(X)\leqslant \xi$.  

$(iv)$ This is clear from $(i)$-$(iii)$.

\end{proof}

\section{Higher order spreading models and distortion}

As previously mentioned, it is a classical result of R.C. James that neither $c_0$ nor $\ell_1$ is $t$-distortable for any $t>1$ \cite{JamesNonSquare}.  In this section, we aim to show that certain types of $\ell_1$ or $c_0$ structure in a Banach space provide a similar non-distortability result with respect to the notion of $\mcs_\xi$ distortion.  

Let $\mcf\in \mathfrak{S}$.  If $X$ is a Banach space, $K\geqslant 1$, and $p\geqslant 1$, we say a basic sequence $(x_i)$ in $X$ is a $K$-$\ell_p^\mcf$ spreading model if there exist $c,C>0$ so that $cC\leqslant K$ and for any $E\in \mcf$ and any scalars $(a_i)_{i\in E}$, $$c^{-1}\Bigl(\sum_{i\in E} |a_i|^p\Bigr)^{1/p}\leqslant \Bigl\|\sum_{i\in E}a_i x_i\Bigr\|\leqslant C\Bigl(\sum_{i\in E}|a_i|^p\Bigr)^{1/p}.$$  We define $K$-$c_0^\mathcal{F}$ spreading models analogously.  If $\mcf=\mcs_\xi$, we write $\ell_p^\xi$ in place of $\ell_p^{\mcs_\xi}$.  

For $1\leqslant p \leqslant \infty$, we will say $(y_i)$ is a $p$-{\it absolutely convex blocking} of $(x_i)$ if there exists a successive sequence $(E_i)\subset [\nn]^{<\omega}$ and scalars $(a_j)$ so that $(a_j)_{j\in E_i}\in S_{\ell_p}$ and $y_i=\sum_{j\in E_i}a_jx_j$ for all $i\in \mathbb{N}$.  We will say $(y_i)$ is a $p$- $\mcf$-{\it absolutely convex blocking} of $(x_i)$ if it is a $p$- absolutely convex blocking of $(x_i)$ and the sets $(E_i)$ can be taken to lie in $\mcf$.     

We record, without proof, the following collection of remarks concerning $\ell_p^\mcf$ and $c_0^\mcf$ spreading models. 

\begin{rem} 
Let $\mcf, \mcg\in \mathfrak{S}$, $K\geqslant 1$.  
\begin{enumerate}[(i)]\item Any subsequence of a $K$-$\ell_p^\mcf$ or $K$-$c_0^\mcf$ spreading model is one as well.  
\item If $\mcf\subset \mcg$ and $(x_i)$ is a $K$-$\ell_p^\mcg$ spreading model, it is a $K$-$\ell_p^\mcf$ spreading model.  
\item If $(x_i)$ is a $K$-$\ell_p^{\mcf[\mcg]}$ spreading model and $(y_i)$ is a $p$- $\mcg$-absolutely convex blocking of $(x_i)$, then $(y_i)$ is a $K$-$\ell^\mcf_p$ spreading model.  
 
\end{enumerate} 
\label{list}
\end{rem}

The search for $\ell_p^\mcf$ or $c_0^\mcf$ spreading models in a Banach space $X$ is typically not impeded by the requirement from the definition that the sequence be basic. For $\ell_p^\mcf$, $p>1$, or $c_0^\mcf$ spreading models we easily observe the following. 

\begin{rem} 
Suppose $\mcf\in \mathfrak{S}$ contains sets of arbitrarily large cardinality and all singletons.  Suppose $(x_i)_i$, $c,C$ are such that $$c^{-1}\Bigl(\sum_{i\in E} |a_i|^p\Bigr)^{1/p}\leqslant \Bigl\|\sum_{i\in E} a_ix_i\Bigr\|\leqslant C\Bigl( \sum_{i\in E}|a_i|^p\Bigr)^{1/p},$$ for all $E\in \mcf$, scalars $(a_i)_{i\in E}$, and $p>1$.  Then we can find a successive sequence $(E_i)$ in $\mcf$ with $|E_i|= i$.  If $M\in \infin$, then $$\Bigl\|\sum_{j\in M(E_i)} i^{-1}x_j\Bigr\| \leqslant Ci^{1/p}/i \underset{i\to \infty}{\to} 0,$$ hence $(x_i)$ is weakly null.  Since $\mcf$ contains all singletons, the sequence must be seminormalized, and therefore some subsequence of $(x_i)$ is basic. This subsequence is a $K$-$\ell_p^\mcf$ spreading model.  The $K$-$c_0^\mcf$ spreading model is handled similarly.  These hypotheses on $\mcf$ will be satisfied whenever $\mcf=\mcs_\xi$ for $\xi>0$.    
\end{rem}

For $\ell_1^\xi$ spreading models we observe the following.

\begin{rem}
Let $\xi>0$.  An easy induction argument shows that if $M=(2n)_{n\in \mathbb{N}}$, then $\mcs_\xi[\mathcal{A}_2](M)\subset \mcs_\xi$.  If $(x_i)\subset X$, $c,C>0$ are so that for all $E\in \mcs_\xi$ and scalars $(a_i)_{i\in E}$,
$$c^{-1}\sum_{i\in E} |a_i|\leqslant \Bigl\|\sum_{j\in E} a_jx_j\Bigr\|\leqslant C\sum_{j\in E}|a_j|,$$ 
then $(x_i)$ satisfies the same inequalities with $\mcs_\xi$ replaced by $\mcs_\xi[\mathcal{A}_2](M)$.  By passing to the subsequence $(x_i)_{i\in M}$, we can assume the sequence itself satisfies these inequalities when $E\in \mcs_\xi[\mathcal{A}_2]$.  

If some subsequence of the sequence $(x_i)$ is equivalent to the unit vector basis of $\ell_1$, this subsequence is clearly basic.  Otherwise we can use Rosenthal's $\ell_1$ theorem \cite{Rosenthalell1}, pass to a subsequence, and assume $(x_i)$ is weakly Cauchy.  Then $y_i = (x_{2i}-x_{2i+1})/2$ defines an $\mathcal{A}_2$-absolutely convex blocking of $(x_i)$, and $(y_i)$ is weakly null.  Therefore some subsequence of $(y_i)$ is basic and hence a $K$-$\ell_1^\xi$ spreading model. 
\end{rem}

The following proof is a transfinite analogue of a well-known argument due to James \cite{JamesNonSquare} and a sharpening of a result of Judd and Odell \cite[Lemma 6.5]{JuddOdell}.

\begin{theorem} Let $X$ be a Banach space, $\xi< \omega_1$.  If $X$ contains an $\ell_1^{\omega^\xi}$ spreading model, then for any $\varepsilon>0$, $X$ contains a $(1+\varepsilon)$-$\ell_1^{\omega^\xi}$ spreading model.  More precisely, if $(x_i)\subset X$ is an $\ell_1^{\omega^\xi}$ spreading model, we can find a blocking $(y_i)$ of $(x_i)$ which is a $(1+\varepsilon)$-$\ell_1^{\omega^\xi}$ spreading model so that for each $F\in \mcs_{\omega^\xi}$, $\cup_{i\in F} \text{supp}_{(x_j)}y_i \in \mcs_{\omega^\xi}$.   \label{goodconstants}
\end{theorem}

\begin{proof}

We prove the result for $\xi>0$.  The result for $\xi=0$ is somewhat simpler, and involves using $\mathcal{A}_n$ in place of $\mcs_{\xi_n}$, where $\xi_n\uparrow \omega^\xi$ is the sequence of ordinals used to define $\mcs_{\omega^\xi}$.  Note that if $(x_i)$ is a $K$-$\ell_1^{\omega^\xi}$ spreading model, and if $c,C>0$ are as in the definition of $K$-$\ell_1^{\omega^\xi}$ spreading model, we can replace $x_i$ by $C^{-1}x_i$ and assume $(x_i)\subset B_X$, $c=K$, and $C=1$.  We will assume this.  We will prove by induction on $m\in \mathbb{N}$ that if $(x_i^0)\subset B_X$ is a $K$-$\ell_1^{\omega^\xi}$ that there exist a blocking $(x_i^m)\subset B_X$ of $(x_i^0)$ and an ordinal $\zeta_m<\omega^\xi$ so that $(x_i^m)$ is a blocking of $(x_i^{m-1})$, $(x_i^m)$ is a $K^{1/2^m}$-$\ell_1^{\omega^\xi}$ spreading model, and $\text{supp}_{(x_j^0)} x_i^m\in \mcs_{\zeta_m}$.  Of course, the base case $m=0$ is the hypothesis. 

Suppose we have found the blocking $(x_i^{m-1})$ and the ordinal $\zeta_{m-1}<\omega^\xi$.  We say a sequence $(y_i)$ in $X$ has property $P_n$ if for each $E\in \mcs_{\xi_n}$ with $n\leqslant E$ and for all scalars $(a_i)_{i\in E}$, $$K^{1/2^m}\Bigl\|\sum_{i\in E} a_iy_i\Bigr\|\geqslant \sum_{i\in E} |a_i|.$$  Note that if $(y_i)$ has property $P_n$ and if $(z_i)$ is a sequence in $X$ so that $z_i=y_i$ for all $i\geqslant n$, then $(z_i)$ also has property $P_n$.  Also, by the spreading property of the Schreier families, any subsequence of a sequence with property $P_n$ also has property $P_n$.  We consider two cases.  

In the first case, for all $n\in\mathbb{N}$ and for all $M\in [\nn]$, there exists $N\in [M]$ so that $(x_i^{m-1})_{i\in N}$ has property $P_n$.  In this case, we let $M_0=\nn$ and choose recursively $M_1, M_2, \ldots$ so that $M_n\in [M_{n-1}]$ and so that $(x_i^{m-1})_{i\in M_n}$ has property $P_n$.  Moreover, since property $P_n$ is invariant under redefining the first $n-1$ elements of a sequence, we can assume that the first $n-1$ elements of $M_n$ are the same as the first $n-1$ elements of $M_{n-1}$.  If we write $M_n=(m_i^n)$, we let $m_n=m_n^n$ and $M=(m_n)$.  Since $M\in [M_n]$ for all $n\in \nn$, $(x^{m-1}_i)_{i\in M}\subset $ has property $P_n$ for all $n\in \nn$.  We let $x^m_i=x_{m_i}^{m-1}$ and $\zeta_m=\zeta_{m-1}$.  This is clearly the desired sequence.  

In the second case, there exists $M\in [\nn]$ and $n\in \mathbb{N}$ so that no subsequence of $(x^{m-1}_i)_{i\in M}$ has property $P_n$.  By relabeling, we can assume $M=\nn$.  Let $\zeta_m=\zeta_{m-1}+\xi_n$.  Choose by Proposition \ref{Schreier facts} (iii) some $N=(n_i)\in [\nn]$ so that $$\mcs_{\omega^\xi}[\mcs_{\xi_n}](N)\subset \mcs_{\xi_n+\omega^\xi}=\mcs_{\omega^\xi}.$$  Let $$F_i=\text{supp}_{(x_j^0)} x^{m-1}_{n_i}\in \mcs_{\zeta_{m-1}}.$$  Since $(x_i^{m-1})$ is a $K^{1/2^{m-1}}$-$\ell_1^{\omega^\xi}$ spreading model, $(x_i^{m-1})_{i\in N}$ is a $K^{1/2^{m-1}}$-$\ell_1^{\mcs_{\omega^\xi}[\mcs_{\xi_n}]}$ spreading model.  Choose by Proposition \ref{Schreier facts} (iv) some $L\in [\nn]$ so that if $E\in \mcs_{\xi_n}(L)$, then $$\bigcup_{i\in E}F_i\in \mcs_{\zeta_{m-1}+\xi_n}=\mcs_{\zeta_m}.$$  Since no subsequence of $(x_{n_i}^{m-1})_{i\in L}$ has property $P_n$, we can find $E_1<E_2<\ldots$, $E_i\in \mcs_{\xi_n}$ and non-zero scalars $(a_j)$ so that for each $i$, $$\Bigl\|\sum_{j\in L(E_i)} a_j x ^{m-1}_{n_j}\Bigr\|< K^{-1/2^m} \text{\ \ and\ \ }\sum_{j\in L(E_i)} |a_j|=1.$$  Then $y_i=\sum_{j\in L(E_i)} a_jx^{m-1}_{n_j}$ is a $\mcs_{\xi_n}$ absolutely convex blocking of $(x^{m-1}_{n_i})$, which is a $K^{1/2^{m-1}}$-$\ell_1^{\mcs_{\omega^\xi}[\mcs_{\xi_n}]}$ spreading model, $(y_i)$ is a $K^{1/2^{m-1}}$-$\ell_1^{\omega^\xi}$ spreading model.  Then by homogeneity, $(x^m_i)=(K^{1/2^m}y_i)\subset B_X$ is a $K^{1/2^m}$-$\ell_1^{\omega^\xi}$ spreading model. This finishes the inductive step.  

We next choose $n_0\in \mathbb{N}$ so that $K^{1/2^{n_0}}<1+\varepsilon$ and $(\ell_i)=L$ according to Proposition \ref{Schreier facts} (iv) so that if $E\in \mcs_{\omega^\xi}(L)$, $$\cup_{i\in E} \text{supp}_{(x_j^0)} x_i^{n_0} \in \mcs_{\zeta_{n_0}+\omega^\xi}=\mcs_{\omega^\xi}.$$  The announced sequence is $(x^{n_0}_{\ell_i})$.  

The proof for $c_0^{\omega^\xi}$ spreading models is similar, except we reverse the inequalities.  Given $(x_i^0)$, we can assume $c=1$, $C=K$, and $\|x_i^0\|\geqslant 1$ for all $i\in \nn$.  We find successive blockings $(x^m_i)$ and ordinals $\xi_m<\omega^\xi$ so that $\|x_i^m\|\geqslant 1$, for all $E\in \mcs_{\omega^\xi}$ and scalars $(a_i)_{i\in E}$, $$\Bigl\|\sum_{i\in E}a_ix^m_i\Bigr\|\leqslant K^{1/2^m} \max_{i\in E}|a_i|,$$  and so that $\text{supp}_{(x_j^0)}x^m_i\in \mcs_{\xi_m}$.  Suppose $n_0\in \nn$ is chosen so that $K^{1/{2^{n_0}}}<1+\delta$, where $\delta\in (0,1)$ is so small that $(1+\delta)/(1-\delta)<1+\varepsilon$.  Choose $E\in \mcs_{\omega^\xi}$ and scalars $(a_i)_{i\in E}$ so that $\max_{i\in E}|a_i|=1$.  We assume there exists $j\in E$ so that $a_j=1$.  Let $w=\sum_{i\in E} a_i x^{n_0}_i$ and let $w'=a_jx^{n_0}_j-\sum_{i\in E\setminus (j)}a_ix_i^{n_0}$.  Then $$2\leqslant 2\|x^{n_0}_j\| = \|w+w'\| \leqslant \|w\|+\|w'\| \leqslant 1+\delta+ \|w'\|,$$ whence $\|w'\|\geqslant 1-\delta$.  This implies $(x_i^{n_0})$ is a $(1+\varepsilon)$-$c_0^{\omega^\xi}$ spreading model.

\end{proof}

\begin{remark}

We observe that for any $1<p<\infty$, we can replace $1$-absolutely convex blockings with $p$-absolutely convex blockings in the first argument to deduce that if $X$ contains a $K$-$\ell_p^{\omega^\xi}$ spreading model, then for any $\varepsilon>0$, there exists a $K$-$\ell_p^{\omega^\xi}$ spreading model $(x_i)\subset B_X$ so that for each $E\in \mcs_{\omega^\xi}$ and scalars $(a_i)_{i\in E}$, $$\Bigl(\sum_{i\in E}|a_i|^p\Bigr)^{1/p}\leqslant (1+\varepsilon)\Bigl\|\sum_{i\in E} a_ix_i\Bigr\|.$$  But in this case, tight $\ell_p$ upper estimates do not follow as in the $\ell_1$ case.  In fact, the theorem is false in this case, otherwise $\ell_p$ would not be distortable.  Similarly, in the second argument we can replace $\infty$-absolutely convex blockings with $p$-absolutely convex blockings to deduce that if $X$ contains a $K$-$\ell_p^{\omega^\xi}$ spreading model, then for $\varepsilon>0$, there exists a $K$-$\ell_p^{\omega^\xi}$ $(x_i)$ in $X$ so that $\|x_i\|\geqslant 1$ for all $i\in \nn$ and so that for all $E\in \mcs_{\omega^\xi}$ and scalars $(a_i)_{i\in E}$, $$\Bigl\|\sum_{i\in E}a_ix_i\Bigr\|\leqslant (1+\varepsilon)\Bigl(\sum_{i\in E}|a_i|^p\Bigr)^{1/p}.$$

\end{remark}

We now state the following lemma that will be used to prove one of our main theorems. Essentially it states that if $(x_i)$ is either an $\ell_1^{\omega^\xi}$ or $c_0^{\omega^\xi}$ spreading model with respect to two equivalent norms on $X$, we can block the spreading model to improve its constant with respect to one norm {\it without worsening the constant with respect to the other norm}.  

\begin{lemma}
Suppose that $X$ is a Banach space and $(x_i)$ is a $C$-$\ell_1^{\omega^\xi}$ spreading model in $(X, \|\cdot\|)$.  Let $|\cdot|$ be an equivalent norm on $X$. Then for each $\ee>0$ there is a block $(y_i)$ of $(x_i)$  satisfying 
\begin{itemize}
\item[(i)] The sequence $(y_i)$ is a $(1+\ee)-\ell_1^{\omega^\xi}$ spreading model in $(X,|\cdot|)$.
\item[(ii)] The sequence $({\color{blue}\lambda}y_i)$ is a $C-\ell_1^{\omega^\xi}$ spreading model in $(X,\|\cdot\|)$.
\end{itemize}
\end{lemma}

\begin{proof} This is an immediate consequence of Theorem \ref{goodconstants}.  The blocking $(y_i)$ of $(x_i)$ which is a $(1+\varepsilon)$-$\ell_1^{\omega^\xi}$ spreading model in $(X, |\cdot|)$ is so that if $E\in \mcs_{\omega^\xi}$, $\cup_{i\in E} \text{supp}_{(x_j)}y_i\in \mcs_{\omega^\xi}$, which clearly implies that $(y_i)$ is still a $C$-$\ell_1^{\omega^\xi}$ spreading model in $(X, \|\cdot\|)$.

\end{proof}

We use the above lemma to prove the following

\begin{proposition} If $X$ is a Banach space with basis $(e_i)$ and $0 \leqslant \xi< \omega_1$ is such that $X$ contains either an $\ell_1^{\omega^\xi}$ or a $c_0^{\omega^\xi}$ spreading model, then $D_{1+\varepsilon}(X)>\omega^\xi$ for any $\varepsilon>0$.   \label{maincorollary}
\end{proposition}


\begin{proof}
If $\ell_1\hookrightarrow X$, then we reach the conclusion by Proposition $2$ and the fact that $\ell_1$ is not distortable.  So we assume $X$ contains no copy of $\ell_1$. In this case, Remark $2$ implies that if $X$ contains an $\ell_1^\zeta$ spreading model for some $\zeta>0$, it contains one which is weakly null.  Therefore we apply a standard perturbation argument and Theorem $1$ to deduce the existence of a block sequence $(x_i)\subset B_X$  which is a $(1+\delta)$-$\ell_1^{\omega^\xi}$ spreading model, $\delta>0$ to be determined.  Let $|\cdot|$ be an equivalent norm on $X$.  By the previous remark, we can find a blocking $(y_i)$ which is a $(1+\delta)$-$\ell_1^{\omega^\xi}$ spreading model in $(X, \|\cdot\|)$ and a $(1+\delta)$-$\ell_1^{\omega^\xi}$ spreading model in $(X, |\cdot|)$.  Then there exist constants $a,a_0,b,b_0>0$ so that $ab, a_0b_0\leqslant 1+\delta$ so that for $E\in \mcs_{\omega^\xi}$, $(a_i)_{i\in E}$, and $x=\sum_{i\in E}a_iy_i$, $$a^{-1}\sum_{i\in E}|a_i| \leqslant \|x\|\leqslant b\sum_{i\in E}|a_i|$$ and $$
a_0^{-1}\sum_{i\in E}|a_i|\leqslant |x|\leqslant b_0\sum_{i\in E}|a_i|.$$  
Fix $E\in \mcs_{\omega^\xi}$ and $x=\sum_{i\in E}a_iy_i, y=\sum_{i\in E} b_iy_i\in S_X$.  Then $$\frac{|x|}{|y|} \leqslant\frac{b_0\sum_{i\in E}|a_i|}{a_0^{-1}\sum_{i\in E}|b_i|} \leqslant\frac{b_0a\|x\|}{a_0^{-1}b^{-1}\|y\|} \leqslant (1+\delta)^2.$$  With an appropriate choice of $\delta\geq 0$, we reach the conclusion.  

 The proof in the $c_0$ case is similar.  
\end{proof}

While the following is an aside, it is worth observing. It states that not containing an $\ell_1^{\omega^\xi}$ spreading model is a three space property. The proof is very similar to the proof of Theorem \ref{goodconstants}.  

\begin{proposition} Let $X$ be a Banach space, $Y$ a closed subspace, and $\xi< \omega_1$.  Then $X$ contains an $\ell_1^{\omega^\xi}$ spreading model if and only if $Y$ or $X/Y$ does.  \end{proposition} 

\begin{proof}

It is known that $X$ contains a copy of $\ell_1$ if and only if either $Y$ or $X/Y$ does, so assume that none of these three spaces contains a copy of $\ell_1$.  This assumption allows us to use Remark $2$ to deduce that it is sufficient to find a sequence in the unit ball of the appropriate space which satisfies the desired lower estimate.  That is, we do not need this sequence to be basic, since it will have a blocking which is an $\ell_1^{\omega^\xi}$ spreading model.  Again, we include the details only of the $\xi>0$ case.  Let $\xi_n\uparrow \omega^\xi$ be the ordinals used to define $\mcs_{\omega^\xi}$.  

Suppose $X$ contains a $K$-$\ell_1^{\omega^\xi}$ spreading model $(x_i)$.  Again, we assume $c=K$, $C=1$, and $(x_i)\subset B_X$. Let us say $(u_i)\subset X$ has property $P_n$ if for all $n\leqslant E\in \mcs_{\xi_n}$ and scalars $(a_i)_{i\in E}$, $$K_0\Bigl\|\sum_{i\in E} a_i u_i +Y\Bigr\|_{X/Y} \geqslant \sum_{i\in E} |a_i|,$$ where $K_0>K$ is fixed.  As in the proof of Theorem \ref{goodconstants}, we either pass to a subsequence $(x_i)_{i\in N}$ of $(x_i)$ which has property $P_n$ for all $n$ or there exists an $n\in \nn$ and an $\mcs_{\xi_n}$-absolutely convex blocking $(z_i)$ which is also a $K$-$\ell_1^{\omega^\xi}$ spreading model in $X$ and so that $\|z_i+Y\|_{X/Y}<K_0^{-1}$ for all $i$.  In the first case, the sequence $(z_i+Y)\subset B_{X/Y}$ and satisfies the desired lower estimate with constant $K_0$.  In the second case, choose for each $i\in \nn$ some $y_i\in 2B_Y$ so that $\|z_i-y_i\|<K_0^{-1}$ and let $\delta=K^{-1}-K_0^{-1}$.  Then if $E\in \mcs_{\omega^\xi}$ and $(a_i)_{i\in E}$, $$\Bigl\|\sum_{i\in E} a_i y_i\Bigr\|\geqslant \Bigl\|\sum_{i\in E}a_iz_i\Bigr\|-\sum_{i\in E}|a_i|\|z_i-y_i\| \geqslant \delta\sum_{i\in E}|a_i|.$$ The upper $\ell_1$ estimates on $(y_i)$ follow from the fact that $(y_i)\subset 2B_Y$.  

The other direction is trivial.  \end{proof}

It is worth pointing out that an analogous result cannot be stated for $\ell_p^{\omega^\xi}$, $p>1$, or $c_0^{\omega^\xi}$ spreading models. For instance, if $X = \ell_1$, then there exists a subspace $Y$ of $X$ such that the space $X/Y$ is isometric to $\ell_p$ or $c_0$. In particular, $X/Y$ contains an $\ell_p^{\omega^\xi}$ or $c_0^{\omega^\xi}$, for every countable ordinal number $\xi$, however $X$ does not contain any order of $\ell_p$ or $c_0$ spreading models.

\section{Computing the distortion index for certain spaces}

In this section we compute or bound the distortion indices for several spaces. As stated in the introduction, the present paper was inspired by a question of P. Dodos question on MathOverFlow \cite{DodosMO}. Here, we resolve several of the queries found there. In particular, we observe that that $AD(S)=2$ where $S$ is Schlumprecht's space, $AD(X)\leqslant 2$ for any asymptotic $\ell_p$ space $X$ with $1<p$, and for every countable ordinal $\xi$ there is an arbitrarily distortable space $X$ such that $AD(X)>\xi$. Also, as noted by Dodos, for each countable $\zeta \geqslant 1$ there is a mixed Tsirelson space $X_\zeta$ \cite[Chapter 13]{ATMemoirs} (that is a higher order analogue of the asymptotic $\ell_1$ mixed Tsirelson space of Argyros-Deliyanni \cite{AD}) that is arbitrarily distortable and asymptotic $\ell_1^\zeta$. For $\zeta=\omega^\xi$, Proposition \ref{maincorollary} implies that $AD(X_\xi)>\xi$. We also present the examples $(\Xox)_{\xi <\omega_1}$ which were first introduced in the paper \cite{ABM-Illinois} by S.A. Argyros, the first and third authors. For these spaces we are able to calculate the exact index; namely, we show that $AD(X)=\omega^\xi+1$ for any block sequence $X$ of $\Xox$. As these spaces have the property that in every subspace there are exactly two spreading models $c_0$ and $\ell_1$, this answers, in the negative, the conjecture of Dodos which asked whether an arbitrarily distortable space with $AD(X)>1$ contains an asymptotic $\ell_p$ space. We note that $S$ also serves as a counterexample to this conjecture.

\subsection{Tsirelson space $T$, Schlumprecht space $S$, and asymptotic $\ell_p$ spaces}
Let $T$ denote the Figiel-Johnson Tsirelson space \cite{FigielJohnson,T}. We note that $T$ is asymptotic $\ell_1$. This implies that any normalized block sequence in $T$ is an $\ell_1^1$ spreading model.  Therefore by Proposition 3 we can deduce that $D_{1+\varepsilon}(Y)>1$ for all $\varepsilon>0$ and any block subspace $Y$ of $T$.  In \cite{OdellTJWagner}, it is shown that Tsirelson space is $(2-\ee)$-distortable for every $\ee>0$ (also see \cite[pgs. 1343 - 1343]{OS-Handbook}). This proof roughly goes as follows: For every $n \inn $ one can find $\ell_1^n$ averages with good constants and {\it rapidly increasing sequences} of $\ell_1^n$ averages -- typically called RIS vectors.  The $\ell_1^n$ averages in any block sequences have supports in $\mcs_1$, while the RIS vectors can be realized with supports in $\mcs_2$. Since these vectors witness the appropriate $2-\ee$ distortion, for any $1<t<2$, $D_t(T)=2$.  

Similarly, for the space $S$ we have $D_{1+\varepsilon}(S)>1$, since $S$ contains an $\ell_1^1$ spreading model \cite{KL}.  But $S$ can be arbitrarily distorted by $\ell_1^n$ averages and RIS vectors, which can again be found in any normalized block with supports in $\mcs_1$ and $\mcs_2$, respectively.  Therefore we deduce that $AD(S)=2$.  

We also observe that in their famous solution to the distortion problem, Odell and Schlumprecht \cite{OSdistortion} used a generalization of the Mazur map to prove that for $1<p<\infty$, $\ell_p$ is arbitrarily distortable.  The construction involved using appropriate pointwise products of sequences of RIS vectors in $S$ and norming functionals in $S^*$ to construct sequences in $\ell_1$ and then transport them to $\ell_p$ to construct the norms which witness the distortion.  Since the generalization of the Mazur map preserves supports, the processes of taking pointwise products and of taking images under this generalization can only reduce supports.  This means $AD(\ell_p)\leqslant 2$.  Moreover, Maurey's proof \cite{Maureydistortion} that for $1<p<\infty$, asymptotic $\ell_p$ spaces are arbitrarily distortable uses a process similar to that of Odell and Schlumprecht, and yields the same conclusion: If $1<p$, and if $X$ is an asymptotic $\ell_p$ Banach space, $AD(X)\leqslant 2$.  

\subsection{The spaces $(\Xox)_{\xi < \omega_1}$}

The rest of the paper is dedicated to defining and providing the relevant facts concerning the spaces $\Xox$ each $\xi$ with $1 \leqslant \xi <\omega_1$ the space $\Xox$ is reflexive with a 1-unconditional basis. In \cite{OS-Krivine}, Odell and Schlumprecht introduced the method, a form of which used to construct the space $\Xox$. In this paper they construct a space having the property that every unconditional basic sequence is finitely block represented in every subspace. 

In \cite{ABM-Illinois}, a thorough study of the spaces $\X$ is undertaken. Here it is shown that the spaces are quasi-minimal and every subspace admits only $c_0$ and $\ell_1$ spreading models. In a subsequent paper \cite{AM-PLMS}, S.A. Argyros and the third author use these spaces to provide the first reflexive spaces so that every operator on a subspace has a non-trivial invariant subspace. In \cite{ABM-Illinois}, the spaces $\Xxi$ are introduced for countable $\xi$, however, many of the properties were not proved. The next proposition includes the properties of $\Xox$ that we  need to compute $AD(\Xox)$. In the final section we give the definition of $\Xox$ and prove this proposition. 

\begin{proposition}
If $(x_i)$ is a normalized block sequence, either there exists $M\in \infin$ so that $(x_i)_{i\in M}$ is an $\ell_1^{\omega^\xi}$ spreading model or for any $\varepsilon>0$, there exists $M\in [\nn]$ so that $(x_i)_{i\in M}$ is a $(1+\varepsilon)$-$c_0^1$ spreading model. 
In particular we have:
\begin{itemize}
\item[(i)]If $(x_i)_{i\in M}$ is a $c_0^1$ spreading model, there exists a sequence $(E_i)$ of successive elements of $\mcs_1$ so that $y_i=\sum_{j\in E_i}x_j$ is an $\ell_1^{\omega^\xi}$ spreading model.
\item[(ii)]If $(x_i)_{i\in M}$ is an $\ell_1^{\omega^\xi}$ spreading model and $\varepsilon>0$, there exists a sequence $(E_i)$ of successive elements of $\mcs_{\omega^\xi}$ and scalars $(a_j)$ so that $y_i = \sum_{j\in E_i}a_j x_j$ is a $(1+\varepsilon)$-$c_0$ spreading model. 
\end{itemize}
\label{properties of the space}
\end{proposition}

\begin{theorem} If $X$ is any block subspace of $\Xox$, then $AD(X)=\omega^\xi+1$.  \end{theorem}

\begin{proof}

$AD(X)>\omega^\xi$ follows from Proposition \ref{properties of the space} and Proposition \ref{maincorollary}, since the block subspaces of $\Xox$ each contain block sequences which are $\ell_1^{\omega^\xi}$ spreading models.    

For $n\in \nn$, define $|\cdot|_n$ on $\Xox$ by $$|x|_n = \sup\Bigl\{\sum_{i=1}^n \|I_ix\|: I_1<\ldots<I_n, I_i \text{\ an interval}\Bigr\}.$$  Clearly $\|\cdot\|\leqslant |\cdot|_n \leqslant n\|\cdot\|$. We will show that for $\delta>0$, $|\cdot|_n$ $(n-\delta)$-$\mcs_{\omega^\xi+1}$ distorts $\Xox$, and therefore $(n-\delta)$-$\mcs_{\omega^\xi+1}$ distorts any block subspace as well.  

We claim that for any block sequence $(x_i)\subset \Xox$, any $\varepsilon>0$, and any $k\in \nn$, there  exist $E\in \mcs_{\omega^\xi+1}$ and block sequences $(y_i)_{i=1}^k$, $(z_i)_{i=1}^n\subset [x_i]_{i\in E}$ so that $(y_i)_{i=1}^k, (z_i)_{i=1}^n$ are $(1+\varepsilon)$-equivalent to the $\ell_1^k$ and $\ell_\infty^n$ bases, respectively.  First we will show how this claim finishes the proof, and then we will return to the claim. We can assume that $\|y_i\|\leqslant 1$ and $\|z_i\|\geqslant 1$ for all $i$.  Suppose we have found the indicated $(y_i)_{i=1}^k, (z_i)_{i=1}^n$.  Let $\overline{y}=\sum_{i=1}^k y_i, y=\overline{y}/\|\overline{y}\|$, $ \overline{z}=\sum_{i=1}^n z_i, z=\overline{z}/\|\overline{z}\|\in [x_i]_{i\in E}$.  Then $$\|\overline{z}\|\leqslant 1+\varepsilon, \text{\ \ and\ \ }|\overline{z}|_n\geqslant \sum_{i=1}^n \|z_i\| \geqslant n,$$ 
$$\|\overline{y}\|\geqslant k/(1+\varepsilon), \text{\ \ and\ \ }|\overline{y}|_n\leqslant (k+2n).$$ Therefore $$|z|_n/|y|_n \geqslant \frac{n}{(1+\varepsilon)^2}\frac{k}{k+2n}.$$  Since $\varepsilon$ and $k$ were arbitrary, this gives the conclusion.  

We return to the claim.  We assume $k\geqslant n$.  Let $(x_i)$ be a block sequence in $\Xox$.  By Proposition \ref{properties of the space}, we can choose $M\in [\nn]$ so that $(x_i)_{i\in M}$ is either a $(1+\varepsilon)$- $c_0^1$ spreading model or an $\ell_1^{\omega^\xi}$ spreading model.  

Suppose $(x_i)_{i\in M}$ is a $(1+\varepsilon)$- $c_0^1$ spreading model.  If $E\in \mcs_1\cap [M]^{<\omega}$ is any set with $|E|\geqslant n$, we can clearly find $(z_i)_{i=1}^n \subset [x_i]_{i\in E}$.  Choose $(n_i)=N\in [M]$ so that $\mcs_{\omega^\xi}[\mcs_1](N)\subset \mcs_{\omega^\xi}$.  By Proposition \ref{properties of the space}, we can find $E_1<E_2<\ldots$, $E_i\in \mcs_1$ so that if $u_i=\sum_{j\in E_i}x_{n_j}$, $(u_i)$ is an $\ell_1^{\omega^\xi}$ spreading model.  By Theorem \ref{goodconstants}, we can find $F_1<F_2<\ldots$, $F_i\in \mcs_{\omega^\xi}$ and non-zero scalars $(a_j)$ so that if $y_i=\sum_{j\in F_i}a_ju_j$, $(y_i)$ is a $(1+\varepsilon)$-$\ell_1^{\omega^\xi}$ spreading model.  Then $$\text{supp}_{(x_j)} y_i = N(\text{supp}_{(x_j)_{j\in N}} y_i) = N\Bigl(\bigcup_{\ell\in F_i} E_\ell\Bigr)\in \mcs_{\omega^\xi}[\mcs_1](N)\subset \mcs_{\omega^\xi}.$$  Choose $k\leqslant i_1<\ldots<i_k$ and let $E=\cup_{j=1}^k \text{supp}_{(x_\ell)}(y_{i_j})\in \mcs_{\omega^\xi+1}$.  Since $(y_{i_j})_{j=1}^k\subset [x_i]_{i\in E}$ and since $|E|\geqslant n$, this finishes the first case.  

Suppose $(x_i)_{i\in M}$ is an $\ell_1^{\omega^\xi}$ spreading model.  We choose according to Theorem \ref{goodconstants} some $F_1<F_2<\ldots$, $F_i\in \mcs_{\omega^\xi}$ and non-zero scalars $(a_j)$ so that if $y_i=\sum_{j\in F_i} a_jx_j$, $(y_i)\subset B_X$ is a $(1+\varepsilon)$-$\ell_1^{\omega^\xi}$ spreading model.  By Proposition \ref{properties of the space}, we choose $E_1<E_2<\ldots$, $E_i\in \mcs_{\omega^\xi}$ and scalars $(b_j)$ so that if $z_i=\sum_{j\in E_i} b_jx_j$, $(z_i)$ is a $(1+\varepsilon)$- $c_0^1$ spreading model.  Then we choose $i_1<\ldots <i_k$ and $m_1<\ldots <m_n$ so that $$n+k \leqslant F_{i_1}<\ldots <F_{i_k}< E_{m_1}<\ldots < E_{m_n}$$ and set $$E=F_{i_1}\cup\ldots \cup F_{i_k}\cup E_{m_1}\cup\ldots \cup E_{m_n}\in \mcs_{\omega^\xi+1}.$$  Then $(y_{i_j})_{j=1}^k, (z_{m_j})_{j=1}^n \subset [x_i]_{i\in E}$ are clearly the desired blocks.  
\end{proof}

We conclude this section by stating a few problems that are open to us.

\begin{problem}
Does there exist a space $X$ so that $AD(X)=1$?
\end{problem}

\begin{problem}
Construct an arbitrarily distortable space $X$ space admitting neither a $c_0$ nor an $\ell_1$ spreading model so that $AD(X)>1$. This would perhaps reveal a new method of achieving lower bounds for $AD(X)$, which would not involve these spreading models.  \label{our problem}
\end{problem}

\begin{problem} [Gowers]
Is $AD(\ell_2)>1$? \label{Gowers}
\end{problem}

\section{The spaces $(\Xox)_{\xi < \omega_1}$}

Below we define the space $\Xox$ for a countable  ordinal $\xi$.

\begin{notation} Let $G\subset c_{00}$. If a vector $\al_0 \in G$ is of
the form $\al_0 = \frac{1}{\ell}\sum_{q=1}^df_q$, for some $\ell \inn$, $f_1 <
\ldots < f_d \in G,  d\leqslant \ell$ and $2 \leqslant$, then
$\al_0$ will be called an $\al$-average of size $s(\al_0) = \ell$. Notice that the size is not uniquely defined, however, this will not cause a problem.

Let $k\inn$. A finite sequence $(\al_q)_{q=1}^d$ of
$\al$-averages in $G$ will be called $\mathcal{S}_{\omega^\xi}$ admissible if
$\al_1 <\ldots < \al_d$ and $\{\min\supp\al_q:
q=1,\ldots,d\}\in\mathcal{S}_{\omega^\xi}$.

A finite or infinite sequence $(\al_q)_q$ of $\al$-averages in $G$ will be called
{\it very fast growing} if $\al_1 < \al_2 <\ldots$, $s(\al_1) < s(\al_2)
<\cdots$ and $s(\al_q)
> \max\supp\al_{q-1}$ for $1<q$.

If a vector $g\in G$ is of the form $g = \sum_{q=1}^d \al_q$ for
an $\mathcal{S}_{\omega^\xi}$-admissible and very fast growing sequence
$(\al_q)_{q=1}^d\subset G$, then $g$ will be called a
\schr.
\end{notation}

\subsection*{The norming set} Inductively construct a set
$W\subset c_{00}$ in the following manner.
Set $W_0 = \{\substack{+\\[-2pt]-} e_n\}_{n\inn}$. Suppose that $W_0,\ldots,W_m$ have been constructed. Define:
\begin{equation*}
W_{m+1}^\alpha = \big\{ \frac{1}{\ell}\sum_{q=1}^df_q:\quad f_1
< \ldots < f_d \in W_m, \ell\geqslant 2, \ell\geqslant d\big\}
\end{equation*}
\begin{equation*}
W_{m+1}^S\! =\! \big\{ g = \sum_{q=1}^d \al_q:\;
\{\al_q\}_{q=1}^d\subset W_m\; \mcs_{\omega^\xi} \text{-admissible and very fast
growing}\big\}
\end{equation*}
Define $W_{m+1} = W_{m+1}^\alpha\cup W_{m+1}^S\cup W_m$ and $W =
\cup_{m=0}^\infty W_m$.

For $x\in c_{00}$ define $\|x\| = \sup\{f(x): f\in W\}$ and $\Xox =
\overline{(c_{00}(\mathbb{N}),\|\cdot\|)}$. Evidently, $\Xox$ has a
1-unconditional basis. Further properties will imply that $\Xox$ is reflexive.

One may also describe the norm on $\Xox$ with an implicit formula.
For $j\inn, j\geqslant 2, x\in\Xox$, set $\|x\|_j =
\sup\{\frac{1}{j}\sum_{q=1}^d\|E_qx\|\}$, where the supremum is
taken over all successive finite subsets of the naturals
$E_1<\cdots<E_d, d\leqslant j$. Then by using standard arguments
it is easy to see that
\begin{equation*}
\|x\| = \max\big\{\|x\|_0,\;
\sup\{\sum_{q=1}^d\|E_qx\|_{j_q}\}\big\}
\end{equation*}
where the supremum is taken over all $\mathcal{S}_{\omega^\xi}$ admissible
finite subsets of the naturals $E_1<\cdots<E_k$, such that $j_q
> \max E_{q-1}$, for $q>1$.

\begin{definition}
Let $( x_k )_{k \inn}$ be a block sequence in $\Xox$ and let $\xi_n \uparrow \omega^\xi$ be the ordinal sequence defining $\mathcal{S}_{\omega^\xi}$.

We write 
$\alpha_{<\omega^\xi}((x_i)_{i \inn})=0$ 
if for any $n \inn$, any fast growing sequence $(\alpha_q)_{q \inn}$
of $\alpha$-averages in $W$ and for any $(F_k)_{k \inn}$ increasing sequence of subsets of $\nn$, such
that $(\alpha_q)_{q \in F_k}$ is $\mathcal{S}_{\xi_n}$ admissible, the following holds: For any subsequence $(x_{n_k})_{k \inn}$
of $(x_k)_{k \inn}$ we have $\lim_{k}\sum_{q \in F_k} |\alpha_q (x_{n_k})|=0$. If this is not the case, we write $\alpha_{<\omega^\xi}((x_i)_{i \inn})>0$. \label{omegaindex}
\end{definition}

The above index is used to detect when a given block sequence will admit an $\ell_1^{\omega^\xi}$ spreading model or a $c_0^1$ spreading model. We will need the following characterization. 

\begin{proposition}
Let $\xi$ be a countable limit ordinal and $(x_k)_{k \inn}$ be a block sequence in $\Xox$. The following are equivalent.
\begin{itemize}
\item[(i)] $\alpha_{<\omega^\xi} ((x_k)_{k \inn}) =0$
\item[(ii)] For any $\e >0$ and $n\inn$ there exist $j_0, k_0 \inn$ such that for
any $k \geqslant k_0$, and for any $(\alpha_q)_{q =1}^d$ $\mathcal{S}_{\xi_n}$-admissible and very fast growing sequence
of $\alpha$-averages such that $s(\alpha_q)>j_0$ for $q = 1, \ldots , d$, we have that $\sum_{q=1}^d |\alpha_q(x_k)| <\e$.
\end{itemize}
\label{higherorderindex}
\end{proposition}

\begin{proposition}\label{cause index says so}
Let $( x_i )_{i \inn}$ be normalized block sequence in $\Xox$. Then the following hold:
\begin{itemize}
\item[(i)] If $\alpha_{<\omega^\xi}((x_i)_{i \inn}) >0$ then, by passing to a subsequence, $( x_i )_{i \inn }$
generates an $\ell_1^{\omega^\xi}$ spreading model.
\item[(ii)] If $\alpha_{<\omega^\xi}((x_i)_{i \inn}) =0$ then there is a subsequence of $(x_i)$ that generates an isometric $c_0$ spreading model.
\end{itemize}
\end{proposition}

\begin{proof} 
First we prove (i). By Definition \ref{omegaindex} there is an $d \inn$, $\ee >0$, a very fast growing
sequence of $\alpha$-averages $(\alpha_q)_{q \inn}$ in $W$, and sequence $(F_i)_{i \inn}$
of successive finite subsets such that $(\alpha_q)_{q \in F_i}$ is $\mathcal{S}_{\xi_d}$ admissible and a subsequence of $(x_i)_i$, again denoted by $(x_i)_i$, so that for each $i \inn$ and
 $$\sum_{q \in F_i} |\alpha_q (x_i)|>\e.$$
Since the basis is unconditional, we may assume that $\supp \alpha_q \subset \supp x_i$ for all $q\in F_i$ and $i\inn$. 
Using the Proposition \ref{Schreier facts} (ii) we may find a subsequence $M = (m_i)_i$ so that for each $F \in \mathcal{S}_{\omega^\xi}(M)$ we have $\cup_{i \in F} F_i \in \mathcal{S}_{\xi_d +\omega^\xi} = S_{\omega^\xi}$ (note that $\xi_d +\omega^\xi = \omega^\xi$).

Let $G\in\mathcal{S}_{\omega^\xi}$ and $(\lambda_i)_{i\in G}$ be real numbers. Then $M(G) \in \mathcal{S}_{\omega^\xi}(M)$. Thus $\cup_{i \in M(G)} F_i =\cup_{i \in G} F_{m_i}  \in \mathcal{S}_{\omega^\xi}$. Since $q \leqslant \min \supp \alpha_q$ we have that $\cup_{i\in G}\{\min\supp \alpha_q: q\in F_{m_i}\}\in\mathcal{S}_{\omega^\xi}$ and therefore the sequence $\{\alpha_q :q\in \cup_{i\in G}F_{m_i}\}$ is $\mathcal{S}_{\omega^\xi}$ admissible and very fast growing.

We conclude that the functional $g = \sum_{i\in G}\sgn (\lambda_i)\sum_{q\in F_{m_i}}\alpha_q$ is in the norming set $W$ and hence:
$$\left\|\sum_{i\in G}\la_ix_{m_i}\right\| \geqslant g\left( \sum_{i\in G}\la_ix_{m_i} \right) > \e \sum_{i\in G}|\lambda_i|.$$
$G\in\mathcal{S}_{\omega^\xi}$ and $(\lambda_i)_{i\in G}$ were arbitrary, we conclude that $(x_{m_i})_i$ admits an $\ell_1^{\omega^\xi}$ spreading model.

We now prove (ii).
Let $(\e_i)_{i \inn}$ be a summable sequence of positive reals such that $\e_i>3 \sum_{j>i} \e_j$
for all $i \inn$. Using Proposition \ref{higherorderindex}, inductively choose a subsequence, again denoted
by $(x_i)_{i \inn}$, such that for $i_0 \geqslant 2$ and $j_0=\max\supp x_{i_0-1}$ if $(\alpha_q)_{q=1}^\ell$ is $\mathcal{S}_{\xi_{j_0}}$-
admissible $s(\alpha_1) \geqslant \min\supp x_{i_0}$ then for all $i \geqslant i_0$
\begin{equation}
\sum_{q=1}^\ell |\alpha_q (x_i)| < \frac{\e_{i_0}}{i_0 }. \label{small application}
\end{equation}
We will show that for any $t \leqslant i_1 < \ldots < i_t$, $F \subset\{1, \ldots , t\}$ we have
$$|\alpha_0(\sum_{j \in F} x_{i_j})|< 1+ 2 \e_{i_{\min F}}.$$
whenever $\alpha_0$ is an $\alpha$-average and
$$|g(\sum_{j \in F} x_{i_j})|< 1+ 3 \e_{i_{\min F}}.$$
whenever $g$ is Schreier functional. This implies that item (ii) holds.

For functionals in $W_0$ the above is clearly true. Assume for some $m \geqslant 0$ that above holds for any
$t \leqslant i_1 < \ldots < i_t$ and any functional in $W_m$. In the first case, let  $t \leqslant i_1 < \ldots < i_t$
an $\alpha_0 = \frac{1}{\ell} \sum_{q=1}^d f_q$ with $d \leqslant \ell$ and $\ell \geqslant 2$ be an $\alpha$-average in $W_{m+1}$. 

Set
$$E_1 = \{ q : \mbox{ there exists at most one $j \leqslant  t$ such that } \ran f_q \cap \ran x_{i_j} \not= \emptyset \},$$
and $E_2 = \{1, \ldots , \ell\} \setminus E_1$.  For $q \in E_1$,
we have $|f_q(\sum_{j=1}^n x_{i_j})| \leqslant  1$. Therefore $\sum_{q
\in E_1} | f_q (\sum_{j=1}^n x_{i_j})| \leqslant  \# E_1$.

For  $q $ in $E_2$, let $j_q \in \{1, \ldots, t\}$ be minimum such
that $\ran x_{i_{j_q}} \cap \ran f_q \not= \emptyset$. If $q < q'$
are in $E_2$, $j_q < j_{q'}$.  By the inductive assumption

\begin{equation}
\begin{split}
\sum_{q \in E_2} | f_q ( \sum_{j=1}^t x_{i_j})| &< \sum_{q \in E_2} (1 + 3\e_{i_{j_q}}) \\
& < \# E_2 + 3\e_{i_1} + 3\sum_{j>1} \e_{i_j} < \# E_2 + 4
\e_{i_1}.
\end{split}
\end{equation}

\noindent Therefore
$$|\al_0( \sum_{j=1}^t x_{i_j})|= |\frac{1}{\ell} \sum_{q=1}^d f_q ( \sum_{j=1}^t x_{i_j})|< \frac{\#E_1 +\#E_2 +4 \e_{i_1}}{\ell}\leqslant \frac{d +4 \e_{i_1}}{\ell} \leqslant 1+2\e_{i_1}.$$
The last inequality follows from the fact that $\ell \geqslant 2$.

Let $g \in W_{m+1}$ such that $g= \sum_{q=1}^d \alpha_q$ is a Schreier functional. We
assume without loss of generality that
\begin{equation}
\mbox{ran} g \cap \mbox{ran} x_{i_j} \neq \varnothing \mbox{ for all } j = 1, \ldots t.
\label{wlog}
\end{equation}
Set
\begin{equation*}
q_0 = \min\{q: \max \supp \alpha_q \geqslant \min\supp x_{i_{2}}\}.
\end{equation*}

By definition of $\mathcal{S}_{\omega^\xi}$, $(\alpha_q)_{q=1}^d$ is $S_{\xi_{\min\supp \alpha_1}}$-admissible.
Also, by definition, for $q>q_0$
$$s(\alpha_q) > \max\supp \alpha_{q_0} \geqslant \min \supp x_{i_2}.$$
Using (\ref{wlog})
$$ \min\supp \alpha_1 \leqslant \max \supp x_{i_1}.$$
These facts together allow us to use our initial assumption on the sequence $(x_i)_{i \inn}$ (for $i_0=i_2$) and conclude
that for $j \geqslant 2$
 
 \begin{equation}
 \sum_{q>q_0} |\alpha_q (x_{i_j})| < \frac{\e_{i_2}}{i_2}.
 \end{equation}

Using the fact that $i_2 \geqslant t$, it follows that
\begin{equation}
\label{after q0} \sum_{q>q_0} |\alpha_q (\sum_{j=2}^t x_{i_j})| < \e_{i_{1}}.
\end{equation}

The rest of the proof is separated into two cases.

{\it Case 1:} Assume first that for $q< q_0$, $\alpha_q( \sum_{j=1}^t x_{i_j})=0$. In this case we apply the induction hypothesis
for $\alpha_{q_0}$ and (\ref{after q0}) to get:
$$g(\sum_{j=1}^t x_{i_j})= \alpha_{q_0} (\sum_{j=1}^t x_{i_j}) + \e_{i_{1}} < 1 + 3\e_{i_{1}}.$$ 

{\it Case 2:} Alternatively, if $\alpha_q( \sum_{j=1}^t x_{i_j})\not=0$ for some $q<q_0$ the very fast growing assumption on $(\al_{q})_{q=1}^d$ yields that $s(\alpha_{q_0}) > \min \supp x_{i_{1}}$. 

In this case, since the singleton $\alpha_{q_0}$ is $\mathcal{S}_0$ admissible, we can use (\ref{small application}) to conclude that 
\begin{equation}
|\alpha_{q_0}( \sum_{j=1}^t x_{i_j})| <  \frac{t \e_{i_1}}{i_1} \leqslant \e_{i_1}. \label{it is small}
\end{equation}
In the above we used that $t \leqslant i_1$.
Therefore combining (\ref{it is small}) and (\ref{after q0}) as before we have:
\begin{equation}
|g(\sum_{j=1}^t x_{i_j})|= \sum_{q<q_0} |\alpha_q ( x_{i_1})| + |\alpha_{q_0} (\sum_{j=1}^t x_{i_j})| + \sum_{q>q_0} |\alpha_q (\sum_{j=2}^t x_{i_j})| \leqslant 1+2\e_{i_1}.
\end{equation}
The proposition is now proved.
\end{proof}

\begin{definition}
Let $x = \sum_{i\in F}c_ie_i$ be a vector in $c_{00}(\mathbb{N})$, $\zeta < \xi$ be countable ordinal numbers and $\e>0$. If:
\begin{itemize}

\item[(i)] the coefficients $(c_i)_{i\in F}$ are  non-negative and $\sum_{i\in F} c_i = 1$,

\item[(ii)] the set $F$ is in $\mathcal{S}_\xi$ and

\item[(iii)] for every $G\in \mathcal{S}_\zeta$ we have that $\sum_{i\in G\cap F}c_i  < \e$,

\end{itemize}
then we say that the vector $x$ is a $(\xi,\zeta,\e)$ basic special convex combination (or basic s.c.c.).
\end{definition}

The proof of the next proposition can be found in \cite[Chapter 13, Proposition 12.9]{ATMemoirs}.
\begin{proposition}\label{hey such stuff really exists}
For all countable ordinal numbers $\zeta<\xi$, positive real number $\e$ and infinite subset of the natural numbers $M$, there exists $F\subset M$ and non-negative real numbers $(c_i)_{i\in F}$ such that the vector $x = \sum_{i\in F}c_ie_i$ is a  $(\xi,\zeta,\e)$ basic s.c.c.
\end{proposition}

\begin{definition}
Let $(x_k)_{k=1}^m$ be a finite block sequence in $c_{00}(\mathbb{N})$, $\zeta < \xi$ be countable ordinal numbers and $\e>0$. Let $(c_k)_{k=1}^m$ be non-negative real numbers and set $\phi_k = \min\supp x_k$ for $k=1,\ldots,m$. If the vector $\sum_{k=1}^mc_ke_{\phi_k}$ is a $(\xi,\zeta,\e)$ basic s.c.c., then we shall say that the vector $x = \sum_{k=1}^mc_kx_k$ is a $(\xi,\zeta,\e)$ special convex combination (or s.c.c.).
\end{definition}

\begin{lemma}\label{yes we can}
Let $(x_i)_i$ be a block sequence in $c_{00}(\mathbb{N})$. Then for every $k\inn$ and $\e>0$ there exists $F\in\mathcal{S}_{\omega^\xi}$ and non-negative real numbers $(c_i)_{i\in F}$ such that the vector $y = \sum_{i\in F}c_ix_i$ is a $(\xi_{k+1},\xi_k,\e)$ s.c.c.
\end{lemma}

 The next two lemmas are is from \cite{ABM-Illinois} where they are labelled Lemma 3.3 and Lemma 3.4, respectively.  
 
\begin{lemma}\label{average on blocks}
Let $\al_0$ be an $\al$-average in $W$, $(x_k)_{k=1}^m$ be a
normalized block sequence and $(c_k)_{k=1}^m$ non
negative reals with $\sum_{k=1}^mc_k = 1$. Then if $G_{\alpha_0} = \{k:
\ran\alpha_0\cap\ran x_k\neq\varnothing\}$, the following holds:
\begin{equation*}
\left|\alpha_0\bigg(\sum_{k=1}^m c_k x_k\bigg)\right| < \frac{1}{s(\alpha_0)}\sum_{i\in G_\alpha}c_i + 2\max\{c_i: i\in
G_{\alpha_0}\}.
\end{equation*}
\end{lemma}

\begin{lemma} Let $k \inn$, $x = \sum_{i=1}^mc_ix_i$
be a $(\xi_{k+1},\xi_k,\e)$ s.c.c. with $\|x_i\|\leqslant 1$ for $i=1,\ldots,m$.
Let also $(\al_q)_{q=1}^d$ be a very fast growing and
$\mathcal{S}_{\xi_k}$-admissible sequence of $\al$-averages. Then
the following holds.
\begin{equation*}
\sum_{q=1}^d\bigg|\al_q\bigg(\sum_{i=1}^mc_ix_i\bigg)\bigg| < \frac{1}{s(\al_1)} + 6\e
\end{equation*}
\label{average on scc}
\end{lemma}

The proof of the following follows Lemma \ref{average on scc} and the fact that the sequence $(\xi_k)_k$ used to define $\mathcal{S}_{\omega^\xi}$ are such that $\mathcal{S}_{\xi_k}\subset\mathcal{S}_{\xi_{k+1}}$ for all $k\inn$. 

\begin{corollary}\label{i gotta sleep}
Let $(x_k)_k$ be a bounded block sequence in $\Xox$ and $(y_k)_k$ be further block sequence of $(x_k)_k$, such that each $y_k = \sum_{i\in F_k}c_ix_i$ is a $(\xi_{k+1},\xi_k,\e_k)$ s.c.c. with $\lim_k\e_k = 0$. Then $\al_{<\omega^\xi}((y_j)_j) = 0$.
\end{corollary}

 The following easily implies Proposition \ref{properties of the space}.

 \begin{proposition}\label{everything works the way it is supposed to}
 Let $(x_k)_{k \inn}$ be a normalized block sequence in $\Xox$.
 \begin{itemize}
 
 \item[(i)] If $(x_k)_{k \inn}$ generates a spreading model equivalent to $c_0$,  $(F_k)_{k \inn}$ is a sequence of successive subsets of natural numbers such
 that $F_k \in \mathcal{S}_1$ for $k \inn$ and $\lim_{k \to \infty} \# F_k = \infty$   and $y_k = \sum_{i \in F_k} x_i$, then a subsequence $(y_{k_n})_n$ of $(y_k)_k$ generates an $\ell_1^{\omega^\xi}$ spreading model.
 
 \item[(ii)] Suppose $(x_k)_{k \inn}$ generates an $\ell_1^{\omega^\xi}$ spreading model, then there exists  a sequence $(F_k)_{k \inn}$ of successive subsets of natural numbers such  that $F_k \in \mathcal{S}_{\omega^\xi}$ for $k \inn$ and normalized vectors $y_k\in\sspan\{x_i: i\in F_k\}$ such that $(y_k)_k$ generates a $c_0$ spreading model. 
 
 \end{itemize}
 \end{proposition}
 
  \begin{proof}
The proof of (i) is identical to that of Proposition 3.14 in \cite{ABM-Illinois}. 

We prove only (ii). By Lemma \ref{yes we can} we can find  $\{F_k\}$a sequence of successive subsets of $\mathbb{N}$  with $F_k \in \mathcal{S}_{\omega^\xi}$ and seminormalized vectors $y_k^\prime\in\sspan\{x_i: i\in F_k\}$ such that $y^\prime_k$ is an $(\xi_{k+1},\xi_k,\e_k)$ s.c.c. with $\lim_{k \to \infty} \e_k=0$. Then 
$(y_k^\prime)_k$ satisfies the assumptions of Corollary \ref{i gotta sleep}. By normalizing the sequence $(y_k^\prime)_k$ and applying the second statement of Proposition \ref{cause index says so}, we obtain the desired sequence $(y_k)_k$. 
\end{proof}

\end{document}